\documentclass[11pt]{amsart}
\usepackage{amsmath}
\usepackage[latin1]{inputenc}
\usepackage{graphicx}
\usepackage{mathrsfs}

\usepackage{a4wide}
\usepackage{latexsym, amsmath, amsfonts, amsthm, amssymb}

\usepackage{pstricks}
\usepackage{latexsym
}

\newtheorem{Theor}{Theorem}[section]

\newtheorem{rmk}{Remark}[section]

\newcommand{\R}{\mathbb R}

\begin{document}

\title[]{On the Continuity of Center-Outward Distribution\\ and Quantile Functions}
\author[A. Figalli]{Alessio Figalli}
\address{ETH Z\"urich, Dept. Mathematics, R\"amistrasse 101, 8092 Z\"urich, Switzerland.}
\email{alessio.figalli@math.ethz.ch} 
 

\begin{abstract} 
To generalize the notion of distribution function to dimension $d\geq 2$, in the recent papers \cite{CGHH17,Hal17} the authors propose a concept  of {\it center-outward distribution function} based on optimal transportation ideas,
and study the inferential properties of the
corresponding {\it center-outward quantile function}. A crucial tool needed in \cite{Hal17} to derive the desired inferential properties is the continuity and invertibility for the 
center-outward quantile function outside the origin, as this ensures the existence of closed and nested {\it quantile contours}. The aim of this paper is to prove such a 
continuity and invertibility result.
 \end{abstract}

 \dedicatory{To Carlo Sbordone, for his 70th birthday.}

\maketitle

 



\section{Introduction}\label{introintro}
\setcounter{equation}
{0}

Starting with dimension $d=2$, the traditional definition of a distribution function, based on marginal orderings, is unsatisfactory on many counts. Indeed, the ranks induced by its empirical counterpart do not enjoy the properties that make traditional (univariate) ranks a successful tool of inference, while the corresponding quantile function---in the Lebesgue-absolutely continuous case, the inverse of the distribution function---does not exhibit the equivariance behaviour one is expecting from a quantile; see~\cite{CR01,Hal17}.
 This fact, which results from the absence of a canonical ordering of $\mathbb{R}^d$, has been recognized long ago, and a number of   ingenious alternative definitions---all of them reducing, for dimension $d=1$, to the traditional univariate definition---have been considered in the statistical literature. None of them, however, is preserving the inferential properties of their univariate counterparts; see~\cite{Hal17} for a survey.  
 
Motivated by this lack of a statistically sound definition, first in \cite{CGHH17} and then in 
\cite{Hal17}, the authors proposed a new concept  of {\it center-outward distribution function} based on optimal transportation ideas. The starting point is the fact that, denoting by $F$ the traditional distribution function associated with an absolutely continuous distribution $\mathrm{P}$ on the real line (namely, $F(z)=\mathrm{P}((-\infty, z],\ z\in\mathbb{R}$), then $2F-1$ is {\it pushing $\mathrm{P}$ forward} to the uniform distribution $\mathrm{U}_1$ over $[-1,1]$, that one can interpret as the unit ball in $\mathbb{R}$. As the map $2F-1$ is monotone increasing, a classical result in optimal transportation theory \cite{McCann1995} implies that this map is the unique gradient of a convex function mapping $\mathrm{P}$ onto $\mathrm{U}_1$. Note that, whereas $F(z)=\mathrm{P}((-\infty, z])$ yields the probability of nested halflines of the form $(-\infty, z]$, the map $2F-1$ is related to intervals of the form $[z^-,  z^+]$ with $F(z^-)+F(z^+)=1$, whence the terminology {\it center-outward distribution function}. 

The definition of a center-outward distribution function as the unique gradient of function (denoted as ${\bf F}_{\tiny{\pm}}$) pushing $\mathrm{P}$ to the uniform measure over the unit ball readily extends to absolutely continuous distributions over $\mathbb{R}^d$; here, with the name  {\it uniform measure over the unit ball}, we mean the measure $\mathrm{U}_d$ obtained by considering the product of the uniform measure over the unit sphere and the uniform over the  unit interval $[0,1]$. In other words, by the change of variable fomula, 
\begin{equation}
\label{eq:Ud}
\mathrm{U}_d=\mathrm{u}_d(x)dx\qquad \text{with $\mathrm{u}_d(x)=\frac{c_d}{|x|^{d-1}}\mathbf{1}_{B_1}(x)$,} 
\end{equation}
 where $c_d=1/\mathcal H^{d-1}(\mathbb S^{d-1})$ is a dimensional normalizing constant (here $\mathcal H^{d-1}(\mathbb S^{d-1})$ denotes the area of the $(d-1)$-dimensional unit sphere).
 The corresponding {\it center-outward quantile function} is then defined as  the inverse ${\bf Q}_{\tiny{\pm}}:={\bf F}_{\tiny{\pm}}^{-1}$. The properties of ${\bf F}_{\tiny{\pm}}$ have been studied, under the assumption of $\mathrm{P}$ being compactly supported, in \cite{CGHH17}; such assumption has then been relaxed in \cite{Hal17}, where it is shown that  ${\bf F}_{\tiny{\pm}}$ and ${\bf Q}_{\tiny{\pm}}$ (and their empirical counterparts), contrary to all previous concepts that have been proposed in the literature, do enjoy the  inferential properties  expected from distribution and quantile functions in $\mathbb{R}^d$. We refer to \cite{Hal17} for more details.

It is important to observe that, in order to derive these inferential properties, a fundamental fact needed in \cite{Hal17} is the fact that~${\bf Q}_{\tiny{\pm}}$ is a homeomorphism from $B_1\setminus\{0\}$ onto its image. Indeed, this ensures the existence of closed and nested {\it quantile contours}, obtained as the images under ${\bf Q}_{\tiny{\pm}}$ of the nested hyperspheres $\{\partial B_r\}_{0<r<1}$. The objective of this paper is to prove this continuity property needed in \cite{Hal17}.
We note that, although several fundamental results have been obtained in the last 25 years on the regularity of optimal transport maps (see \cite{dPF2,figalliBook} for a survey), the proof of the above-mentioned property is rather delicate, due to the fact that the density of $\mathrm{U}_d$ is singular at the origin whenever $d\geq 2$.

We recall that, given two absolutely continuous probability densities on $\R^d$, there exists a unique transport map that pushes forward one density onto the other and which coincides almost everywhere with the gradient of a convex function (see \cite{McCann1995}).
We shall refer to this map as the {\it optimal transport map}, being implicit that this is the optimal transport map for the quadratic Euclidean cost (see \cite{dPF2} for more details).

Here is our main result.\footnote{
Here and in the sequel, $|E|$ stands for the Lebesgue measure of a Borel set $E$. Also, given $k\geq 0$ and $\alpha \in (0,1)$, we say that a function $f$ belongs to $C^{k,\alpha}_{\rm loc}(\R^d)$ if $f\in C^k(\R^d)$ and its $k$-th derivative is locally $\alpha$-H\"older continuous, namely
$$
\forall\,R>0,\qquad \sup_{x\neq y,\,x,y \in B_R}\frac{|D^kf(x)-D^kf(y)|}{|x-y|^\alpha}<\infty.
$$}

\begin{Theor}\label{maintheo}
Let $\mathrm U_d$ be the uniform measure on $B_1$ (see \eqref{eq:Ud}), and let
$\mathrm{P}=\mathrm{p}(y)dy$ be a probability measure on $\R^d$
satisfying  
$0<\lambda_R\leq \mathrm{p}\leq \Lambda_R$ inside $B_R$ for all $R<\infty$.
Let ${\bf Q}_{\tiny{\pm}}=\nabla \varphi:B_1\to \R^d$ be the unique optimal transport map from $\mathrm U_d$ to $\mathrm{P}$. Then ${\bf Q}_{\tiny{\pm}}$ is a homeomorphism
from $B_1\setminus \{0\}$ onto $\R^d\setminus K$, where $K$ is a compact convex set of Lebesgue measure zero. 

In addition:
\begin{itemize}
\item[(a)] If $\mathrm p\in C^{k,\alpha}_{\rm loc}(\R^d)$ for some $k\geq 0$ and $\alpha \in (0,1)$, then ${\bf Q}_{\tiny{\pm}}:B_1\setminus \{0\}\to \R^d\setminus K$ is a diffeomophism of class $C^{k+1,\alpha}_{\rm loc}$ inside $B_1\setminus\{0\}$, and
\begin{equation}
\label{eq:MA}
{\rm det}\bigl(\nabla {\bf Q}_{\tiny{\pm}}(x)\bigr)=\frac{\mathrm u_d(x)}{\mathrm p\bigl({\bf Q}_{\tiny{\pm}}(x)\bigr)}\qquad \forall\,x \in B_1\setminus\{0\}.
\end{equation}
\item[(b)] If $\mathrm p$ is locally analytic, then ${\bf Q}_{\tiny{\pm}}:B_1\setminus \{0\}\to \R^d\setminus K$ is locally an analytic map.
\item[(c)] If $d=2$ then $K=\{{\bf Q}_{\tiny{\pm}}(0)\}$ and ${\bf Q}_{\tiny{\pm}}$ is a homeomorphism
from $B_1$ onto $\R^2$.
\end{itemize}

\end{Theor}

\begin{rmk}
{\rm
When $\mathrm P=\mathrm p(|y|)dy$ has a radial density, then also the map ${\bf Q}_{\tiny{\pm}}$ is radial (this follows from the uniqueness of the optimal transport map) and the above result is elementary. Indeed, in this case one can explicitly write the optimal map in terms of distribution function, 
and the explicit formula is given by
$$
{\bf Q}_{\tiny{\pm}}(x)={\bf q}_{\tiny{\pm}}(|x|)\frac{x}{|x|},
$$
where ${\bf q}_{\tiny{\pm}}:[0,1]\to [0,\infty)$ is implicitly defined via the identity
$$
s=\mathcal H^{d-1}(\mathbb S^{d-1})\int_{0}^{{\bf q}_{\tiny{\pm}}(s)}r^{d-1}\mathrm{p}(r)\,dr\qquad \forall\,s \in (0,1).
$$
Hence, in this very particular case,
 the conclusions of Theorem \ref{maintheo} hold with $K=\{0\}$, as can easily be checked by direct computations.
}
\end{rmk}

\begin{rmk}
{\rm
As shown in point (c) of Theorem \ref{maintheo}, in the case $d=2$ the map 
${\bf Q}_{\tiny{\pm}}$ is a homeomorphism up to the origin. It is a well-known fact that the Monge-Amp\`ere equation behaves better in dimension two than in higher dimensions (see for instance \cite[Sections 2.5 and 3.2]{figalliBook}), and we do not expect Theorem \ref{maintheo}(c) to be true in dimension $d\geq 3$.  However, finding a counterexample would not be relevant to the problem under investigation (namely, the existence of quantile contours as the images of the sphere $\{\partial B_r\}_{0<r<1}$ under  ${\bf Q}_{\tiny{\pm}}$), so we shall not investigate this question here.
}
\end{rmk}

\section{Proof of Theorem \ref{maintheo}}
To prove our main theorem, we first introduce some notation:
given a convex function $\psi:\R^d\to\R\cup\{+\infty\}$, the
\emph{Monge-Amp\`ere measure} $\mu_\psi$ associated to $\psi$ is defined by 
$$\mu_\psi(A) = |\partial \psi (A)|\qquad \text{for every Borel set $A\subset \R^n$,}$$ 
where
$$
\partial \psi(A):=\bigcup_{x \in A}\partial \psi(x)
$$
and $\partial \psi(x)$ denotes the subdifferential of $\psi$ at $x$, that is
$$
\partial \psi(x):=\bigl\{p\in \R^n\,:\,\psi(z)\geq \psi(x)+\langle p, z-x\rangle \quad \forall\,z \in \R^n\bigr\}.
$$
Note that $\partial\psi(x)$ is a convex set for any $x \in \R^d$. Also,
when $\psi$ is of class $C^2$, the Monge-Amp\`ere measure of $\psi$ is given by $\det (D^2 \psi)dx$ (see \cite[Example 2.2]{figalliBook}).

\begin{proof}[Proof of Theorem \ref{maintheo}]
First of all we note that, since the optimal map ${\bf Q}_{\tiny{\pm}}=\nabla \varphi$ is unique a.e. inside $B_1$, the function $\varphi$ is unique inside $B_1$ up to an additive constant. In particular, with no loss of generality we can set $\varphi(0)=0$.

Outside $B_1$ we simply extend $\varphi$ to be identically equal to $+\infty$ (note that this preserves the convexity of $\varphi$), and 
we consider the Legendre transform of $\varphi$, namely
\begin{equation}
\label{eq:Leg}
\varphi^*(y):=\sup_{x\in \R^n}\bigl\{\langle y, x\rangle -\varphi(x)\bigr\}=\sup_{x\in B_1}\bigl\{\langle y, x\rangle -\varphi(x)\bigr\}\qquad \forall\,y \in \R^d
\end{equation}
(the second equality follows from the fact that $\varphi=+\infty$ on $\R^n\setminus B_1$).
It is well known that $\nabla \varphi^*$ is the optimal transport map from $\mathrm{P}$ onto $\mathrm{U}_d$, and that $\nabla \varphi^*=(\nabla \varphi)^{-1}$ a.e.
(see for instance \cite[Section 6.2.3 and Remark 6.2.11]{AGS}).
In particular, 
\begin{equation}
\label{eq:1Lip}
|\nabla\varphi^*|\leq 1\qquad \text{a.e. in $\R^d$.}
\end{equation}
Note that, because $\mathrm{U}_d$ is supported in $B_1$ which is a convex set, it follows by \cite{caf2} (see also Step 1 in the proof of \cite[Theorem 4.23]{figalliBook}) that
$\varphi^*$ is an Alexandrov solution to the Monge-Amp\`ere equation inside $\R^d$, namely (recall \eqref{eq:Ud})
\begin{equation}
\label{eq:Alex}
\mu_{\varphi^*}(A) = \int_A \frac{\mathrm p(y)}{\mathrm{u}_d(\nabla \varphi^*(y))}\, dy=\frac{1}{c_d}\int_A \mathrm p(y)|\nabla \varphi^*(y)|^{d-1}\, dy\qquad \text{for all $A\subset \R^d$ Borel.}
\end{equation}
Let $K:=\partial \varphi(0)$, and observe that $K$ is a closed convex. Also, since $\varphi$ is locally Lipschitz in a neighborhood of the origin (being a finite convex function inside $B_1$), it follows that $K$ is bounded.

Consider an arbitrary compact set $\mathcal C\subset  \R^d\setminus K$.
First of all, we note that $\partial\varphi^*(\mathcal C)$ is a compact set
(see \cite[Lemma A.22]{figalliBook}).
Also, thanks to \eqref{eq:1Lip} it follows by \cite[Corollary A.27]{figalliBook} that 
\begin{equation}
\label{eq:subdiff B1}
\partial\varphi^*(\mathcal C)\subset\partial\varphi^*(\R^d)\subset \overline{B_1}.
\end{equation}
Furthermore, because $\partial \varphi$ and $\partial \varphi^*$ are inverse to each other  (see \cite[Equation (A.20)]{figalliBook}),
noticing that $\mathcal C\cap \partial\varphi(0)=\mathcal C\cap K=\emptyset$ we deduce 
 that $\partial \varphi^*(\mathcal C)\cap \{0\}=\emptyset$.

In conclusion, this proves that $\partial \varphi^*(\mathcal C)$ is a compact set
satisfying
$\partial \varphi^*(\mathcal C) \subset \overline{B_1}\setminus B_\rho$, for some $\rho>0$ depending on $\mathcal C$.
In particular, this implies that $\rho\leq |\nabla \varphi^*(y)|\leq 1$ for a.e. $y \in \mathcal C$.
Hence, recalling that $\mathrm p$ is locally bounded away from zero and infinity,
thanks to \eqref{eq:Alex} we obtain
\begin{equation}
\label{eq:Alex2}
m_{\mathcal C}|A|\leq \mu_{\varphi^*}(A) \leq M_{\mathcal C} |A|\qquad \text{for all $A\subset {\mathcal C}\subset \subset \R^d\setminus K$ Borel,}
\end{equation}
for some constants $0<m_{\mathcal C}\leq M_{\mathcal C}$.

In order to apply the regularity theory for the Monge-Amp\`ere equation from \cite[Chapter 4]{figalliBook}, we first need to prove that $\varphi^*$ is strictly convex inside $\R^d \setminus K$.
Assume this is false. Then there exists $\hat{y} \in \R^d \setminus K$ and $\hat{q} \in \partial \varphi^*(\hat{y}) \in \overline{B_1}$ such that, if we consider the affine function $\ell(z):=\varphi^*(\hat{y})+\langle \hat{q},z-\hat{y}\rangle$,
the convex set $\Sigma:=\{\varphi^*=\ell\}$ is not a singleton.

Notice that, thanks to \eqref{eq:Alex2},
 \cite[Theorem 4.10]{figalliBook} applies inside any compact subset of $\R^n\setminus K$, so
the convex set $\Sigma$ cannot have any exposed point in $\R^d\setminus K$. 
Hence, the only possibilities are the following:\\
(a) either $\Sigma$ contains an infinite half-line $L$ going from $K$ to infinity;\\
(b) or $\Sigma$ contains an infinite line.

In case (b), \cite[Lemma A.25]{figalliBook} yields that $\partial \varphi^*(\R^d)$ is
contained inside a hyperplane, contradicting the fact that $\nabla\varphi^*$ transports $\mathrm P$ onto the measure $\mathrm U_d$ which is supported on the whole unit ball.

We now need to exclude that case (a) occurs.
The argument in this case is inspired by \cite{FKM}.
With no loss of generality, up to a translation and rotation, we can assume that $0\in K$ and that $L=\{t e_1\,:\,t\geq 0\}$.
By the monotonicity of the subdifferential of convex functions it follows that, given two points $y_1$ and $y_2$,
\begin{equation}
\label{eq:monot}
\langle q_2-q_1,y_2-y_1\rangle\geq 0\qquad \forall\, q_i \in \partial \varphi^*(y_i),\,i=1,2.
\end{equation}
Since $\varphi^*=\ell$ on $L$, it follows that $\hat{q}=\nabla \ell \in \partial\varphi^*(q)$ for all $q \in L$. Hence, applying \eqref{eq:monot}
with $y_1=t e_1 \in L$, $q_1=\hat q$, and $y_2=y$ an arbitrary point in $\R^d$, we get
$$
\langle q-\hat{q},y-t e_1\rangle \geq 0\qquad \forall\,y \in \R^d,\, q \in \partial \varphi^*(y),\,t \geq 0.
$$
Letting $t\to +\infty$ in the above inequality we deduce that
$$
\langle q-\hat{q}, e_1\rangle \leq 0\qquad \forall\, y \in \R^d,\,q \in \partial \varphi^*(y).
$$
Thus, we proved that $\partial \varphi^*(\R^d)$ is contained inside the half-space
$$
H:=\{q\,:\,\langle q-\hat{q}, e_1\rangle\leq 0\}.
$$
Recalling \eqref{eq:1Lip},
this implies that $\nabla \varphi^*$ takes values inside $H\cap \overline{B_1}$ a.e. Since $(\nabla \varphi^*)_\#\mathrm P=\mathrm U_d$ and $\mathrm u_d$ is strictly positive inside $B_1$, it follows that $H\cap B_1=B_1$, which is possible if and only if $\hat q=e_1 \in \partial B_1$
(recall that $\hat q \in \overline{B_1}$, see \eqref{eq:subdiff B1}).

Let $\theta>0$ small, and consider the cone 
$$
\mathscr C_\theta:=\bigl\{y \in B_1\,:\,|y|\leq (1+\theta)\langle y, \hat q\rangle \bigr\}.
$$
Since $0 \in L$ we have $\hat q=\nabla \ell \in \partial \varphi^*(0)$, so applying 
\eqref{eq:monot} with $y_1=0$ and $q_1=\hat{q} \in \partial B_1$, we obtain
$$
\langle \nabla \varphi^*(y)-\hat{q}, y\rangle \geq 0\qquad \forall\,y \text{ where $\varphi^*$ is differentiable}.
$$
Combining this inequality with the definition of $\mathscr C_\theta$ and \eqref{eq:1Lip}, we deduce that
\begin{equation}
\label{eq:inclusion}
\nabla \varphi^*(y) \in \mathscr D_\theta:=
\bigl\{x \in B_1\,:\,\langle x-\hat{q}, \hat{q}\rangle \geq -C_d\theta |x-\hat{q}|\bigr \}\qquad \text{for a.e. }y \in \mathscr C_\theta,
\end{equation}
where $C_d>0$ is a dimensional constant
(cp. \cite[Figure 1]{FKM}).
Because $\mathrm p\geq \lambda_1$ inside $\mathscr C_\theta\subset B_1$ and $\mathrm u_d \leq 2c_d$ inside $\mathscr D_\theta$ for $\theta$ small enough,
it follows by the transport condition $(\nabla \varphi^*)_\#\mathrm P=\mathrm U_d$
and by  \eqref{eq:inclusion} that
$$
2c_d|\mathscr D_\theta|\geq \int_{\mathscr D_\theta}\mathrm u_d(x)\,dx=\int_{(\nabla \varphi^*)^{-1}(\mathscr D_\theta)}\mathrm p(y)\,dy\geq 
\int_{\mathscr C_\theta}\mathrm p(y)\,dy \geq \lambda_1 |\mathscr C_\theta|.
$$
Since $|\mathscr C_\theta|\sim \theta^{n-1}$ and $|\mathscr D_\theta|\sim\theta^{n+1}$, we obtain a contradiction for $\theta$ small enough.
This proves that also case (a) is impossible, thus $\varphi^*$ is strictly convex inside $\R^d\setminus K$.
\smallskip

Since $\varphi^*$ is a strictly convex Alexandrov solution of \eqref{eq:Alex2},
it follows by \cite{caf1,cafC1a,caf2} (see also \cite[Corollary 4.21]{figalliBook}) that $\varphi^*$ is of class $C^{1,\alpha}$ inside $\R^d\setminus K$.
In particular, $\nabla \varphi^*$ is continuous inside $\R^d\setminus K$.\footnote{Actually, since $\partial\varphi^*(K)=\{0\}$, it follows by the continuity of the subdifferential (see \cite[Equation (A.15)]{figalliBook}) that $\nabla\varphi^*$ is continuous on the whole space $\R^d$, with $\nabla\varphi^*(y)=0$ for all $y\in K$.}
Since, by the strict convexity of $\varphi^*$ inside $\R^d\setminus K$,  $\nabla \varphi^*$ is an injective continuous map from $\R^d\setminus K$ onto $B_1\setminus \{0\}$,
we deduce that $\nabla \varphi^*:\R^d\setminus K\to B_1\setminus \{0\}$ is a homeomorphism by the theorem on the invariance of domain.
Recalling that ${\bf Q}_{\tiny{\pm}}=\nabla \varphi=(\nabla \varphi^*)^{-1}$, we conclude that ${\bf Q}_{\tiny{\pm}}$ is a homeomorphism
from $B_1\setminus \{0\}$ onto $\R^d\setminus K$.

To prove that $K$ has Lebesgue measure zero it suffices to observe that $K=(\nabla \varphi^*)^{-1}(\{0\})$, so by the transport condition $(\nabla\varphi^*)_\#\mathrm P=\mathrm U_d$ we get
$$
\int_K \mathrm p(y)\,dy=\int_{\{0\}} \mathrm u_d(x)\,dx=0.
$$
Since $\mathrm p>0$ we conclude that $|K|=0$, as desired.
\\

We now prove the additional statements in the theorem.
\smallskip

$\bullet$ {\it Proof of (a).} We note that if $\mathrm p\in C^{k,\alpha}_{\rm loc}(\R^d)$ for some $k\geq 0$ and $\alpha \in (0,1)$, then \cite[Remark 4.25]{figalliBook}
implies that $\nabla \varphi^*$ is a $C^{k+1,\alpha}_{\rm loc}$ diffeomorphism from $\R^d\setminus K$ onto $B_1\setminus \{0\}$,
hence ${\bf Q}_{\tiny{\pm}}:B_1\setminus \{0\}\to \R^d\setminus K$ is a diffeomophism of class $C^{k+1,\alpha}_{\rm loc}$. 

Since ${\bf Q}_{\tiny{\pm}}|_{B_1\setminus \{0\}}$ is a $C^1$ diffeomorphim, the validity of \eqref{eq:MA} is classical, and we give here a short proof for completeness.
Since $\nabla \varphi^*$ is of class $C^2$ outside $K$, it follows by \cite[Example 2.2]{figalliBook} and \eqref{eq:Alex} that
$$
\int_A{\rm det}\bigl(D^2\varphi^*(y)\bigr)\,dy= \int_A \frac{\mathrm p(y)}{\mathrm{u}_d(\nabla \varphi^*(y))}\, dy\qquad \text{for all $A\subset (\R^d\setminus K)$ Borel.}
$$
By the arbitrariness of $A$, this yields
$$
{\rm det}\bigl(D^2\varphi^*(y)\bigr)=
\frac{\mathrm p(y)}{\mathrm{u}_d(\nabla \varphi^*(y))}\qquad \forall\,y \in \R^d\setminus K.
$$
Since ${\bf Q}_{\tiny{\pm}}=\nabla\varphi=(\nabla\varphi^*)^{-1}$,
for any $x \in (\nabla\varphi)^{-1}(\R^d\setminus K)=B_1\setminus\{0\}$
we obtain
$$
{\rm det}\bigl(\nabla {\bf Q}_{\tiny{\pm}}(x)\bigr)={\rm det}\bigl(D^2\varphi(x)\bigr)
=\frac{1}{{\rm det}\bigl(D^2\varphi^*(\nabla\varphi(x))\bigr)}=\frac{\mathrm{u}_d(x)}{\mathrm p(\nabla\varphi(x))}=\frac{\mathrm{u}_d(x)}{\mathrm p({\bf Q}_{\tiny{\pm}}(x))}.
$$
This proves \eqref{eq:MA}, concluding the proof of (a).

\smallskip

$\bullet$ {\it Proof of (b).} It follows by \cite[Proposition A.43]{figalliBook} that the Monge-Amp\`ere equation is uniformly elliptic on $C^2$ solutions. Hence, by the classical analytic regularity of solution to uniformly elliptic PDEs with analytic data \cite{morrey}, if  $\mathrm p$ is locally analytic then so is ${\bf Q}_{\tiny{\pm}}$.

\smallskip

$\bullet$ {\it Proof of (c).} 
We now focus on the case $d=2$. In this part we shall use coordinates $x=(x_1,x_2)$ and $y=(y_1,y_2)$ to denote points in $\R^2$.

Assume by contradiction that $K$ is not a point.
Since $K$ is a compact convex set of Lebesgue measure zero it must be a segment, say $K=[-1,1]\times \{0\}$.
With no loss of generality, we can assume that $\varphi(0)=0$.
Recalling that $K=\partial\varphi(0)$, this implies that
\begin{equation}
\label{eq:phi mod x}
\varphi(x_1,x_2)\geq |x_1|\qquad \forall\,x=(x_1,x_2)\in \R^2.
\end{equation}
Also, since $\varphi:B_1\to \R$ is convex, there exists a constant $R>0$ such that
\begin{equation}
\label{eq:phi Lip}
|\nabla\varphi(x)|\leq R\qquad \forall\,x \in B_{1/2}.
\end{equation}
Now, given $\delta\in (0,1/4]$, we define
$$
h_\delta:=\varphi(0,\delta),\qquad 
\ell_\delta:=\frac{h_\delta}{\delta}.
$$
Note that, because $\varphi$ is differentiable in the $x_2$ variable, $\ell_\delta\to 0$ as $\delta\to 0$. 

Given $R$ as in \eqref{eq:phi Lip}, we set
$$
\mathcal R_\delta:=[-h_\delta,h_\delta]\times [0,(1+R)\delta].
$$
With this definition, thanks to \eqref{eq:phi mod x} and \eqref{eq:phi Lip} we can apply \cite[Lemma 2.3]{FL} to deduce that
$$ \partial\varphi(\mathcal R_\delta)\supset [-1/2,1/2]\times \Bigl[0,\frac{\ell_\delta}{2(1+R)}\Bigr].
$$
Since $\nabla\varphi={\bf Q}_{\tiny{\pm}}$ is differentiable outside the origin and $\partial\varphi(0)=[-1,1]\times\{0\}$, this implies that
\begin{equation}
\label{eq:subdiff phi R}
\nabla \varphi(\mathcal R_\delta)\supset
[-1/2,1/2]\times \Bigl(0,\frac{\ell_\delta}{2(1+R)}\Bigr].
\end{equation}
Hence, by the transport condition $(\nabla \varphi)_\#\mathrm U_2=\mathrm P$ and because $\nabla \varphi(\mathcal R_\delta)\subset \nabla\varphi(B_{1/2})\subset B_R$ (see \eqref{eq:phi Lip}), we obtain
\begin{equation}
\label{eq:transport R}
\int_{\mathcal R_\delta}\mathrm u_2(x)\,dx\geq \int_{-1/2}^{1/2}\biggl(\int_0^{\frac{\ell_\delta}{2(1+R)}} \mathrm p(y)\,dy_2\biggr)\,dy_1\geq \frac{\lambda_R\ell_\delta}{2(1+R)},
\end{equation}
where we used that $\mathrm p\geq \lambda_R$ inside $B_R$. Noticing that
$$
\mathrm u_2(x_1,x_2)=\frac{1}{2\pi}\frac{1}{\sqrt{x_1^2+x_2^2}},
$$
it follows that (recall that $\frac{h_\delta}{\delta}=\ell_\delta \to 0$ as $\delta \to 0$)
\begin{align*}
\int_{\mathcal R_\delta}\mathrm u_2(x)\,dx &=
\frac{1}{2\pi}\int_{-h_\delta}^{h_\delta}d x_1\int_0^{(1+R)\delta}\frac{dx_2}{\sqrt{x_1^2+x_2^2}}=
\frac{1}{2\pi}\int_{-h_\delta}^{h_\delta}dx_1\int_0^{(1+R)\frac{\delta}{|x_1|}}\frac{ds}{\sqrt{1+s^2}}\\
&\leq C_{R}\int_{-h_\delta}^{h_\delta}\log\biggl(\frac{\delta}{|x_1|}\biggr)dx_1= 2 C_R h_\delta \bigl(|\log\ell_\delta|+1\bigr),
\end{align*}
for some constant $C_R$ depending on $R$.
Combining this bound with \eqref{eq:transport R}, we get 
$$
\ell_\delta\leq \hat C_R h_\delta |\log\ell_\delta|\quad \Rightarrow\quad \frac{1}{\delta}
\leq \hat C_R|\log\ell_\delta|=\hat C_R\left|\log\Bigl(\frac{h_\delta}{\delta}\Bigr)\right|.
$$
Recalling that $h_\delta=\varphi(0,\delta)=o(\delta)$, this proves that
$$
\varphi(0,\delta)\leq \delta e^{-c_R/\delta}\qquad \forall\, \delta \in [0,1/4],
$$
where $c_R:=1/\hat C_R$.
Analogously, repeating the above argument with $h_\delta=\varphi(0,-\delta)$ we 
obtain 
$\varphi(0,-\delta)\leq \delta e^{-c_R/\delta}$, therefore
$$
\varphi(0,x_2)\leq |x_2| e^{-c_R/|x_2|}\qquad \forall\,x_2\in [-1/4,1/4].
$$
By the definition of $\varphi^*$ (see \eqref{eq:Leg}), this implies that
\begin{equation}
\label{eq:varphi*}
\begin{split}
\varphi^*(y_1,y_2)&\geq \sup_{|x_2|\leq 1/4}\bigl\{x_2y_2-\varphi(0,x_2)\bigr\}\geq 
\sup_{|x_2|\leq 1/4}\bigl\{x_2y_2-|x_2| e^{-c_R/|x_2|}\bigr\}\\
&\geq c_R'
\frac{|y_2|}{\bigl|\log|y_2|\bigr|}\qquad \forall\,(y_1,y_2) \in \R \times[-1/2,1/2]
\end{split}
\end{equation}
for some constant $c_R'>0$.

At this moment one may conclude as follows: by \eqref{eq:Alex}, the Monge-Amp\`ere measure of $\varphi^*$ is bounded from above. So, thanks to \eqref{eq:varphi*}, we can apply \cite[Theorem 1.4]{mooney} to conclude that $K=[-1,1]\times\{0\}$ contains the infinite line $\R\times \{0\}$, thus providing the desired contradiction.
For completeness we provide here an alternative
self-contained proof, that we believe to have its own interest.

Consider the 
sets
$\mathcal U_k:=[-1/2,1/2]\times [2^{-(k+1)},2^{-k}]\subset \R^2$.
By the transport condition $(\nabla\varphi^*)_\#\mathrm P=\mathrm U_2$
one has (recall that $\mathrm p\geq \lambda_1$ inside $B_1$)
$$
{\rm det}\bigl(D^2\varphi^*(y)\bigr)=\frac{\mathrm p(y)}{\mathrm u_2(\nabla \varphi^*(y))}=2\pi \,\mathrm p(y)|\nabla \varphi^*(y)| \geq 2\pi \lambda_1
|\nabla \varphi^*(y)|\qquad \text{for a.e. $y \in \mathcal U_k$.}
$$
Hence, arguing as in \cite{caf-note}, it follows by the arithmetic-geometric inequality that\footnote{To rigorously justify the inequalities
$$
\int_{\mathcal U_k}\partial_{y_1y_1}\varphi^*(y)\,dy\leq \int_{\partial\mathcal U_k}\partial_{y_1}\varphi^*(y)\,\nu_1,\qquad\int_{\mathcal U_k}\partial_{y_2y_2}\varphi^*(y)\,dy\leq \int_{\partial\mathcal U_k}\partial_{y_2}\varphi^*(y)\,\nu_2,
$$
one can either use that any pointwise pure second derivative of a convex function is bounded from above by the corresponding distributional derivative, or prove the inequalities for smooth functions and then argue by approximation.}
\begin{align*}
2\cdot(2\pi \lambda_1)^{1/2}\int_{\mathcal U_k}|\nabla \varphi^*(y)|^{1/2}\,dy &
\leq 2\int_{\mathcal U_k}{\rm det}\bigl(D^2\varphi^*(y)\bigr)^{1/2}\,dy\\
&\leq \int_{\mathcal U_k}\Bigl(t\,\partial_{y_1y_1}\varphi^*(y)+\frac{1}{t}\,\partial_{y_2y_2}\varphi^*(y)\Bigr)\,dy\\
&\leq t\int_{\partial\mathcal U_k}\partial_{y_1}\varphi^*(y)\,\nu_1+\frac{1}{t}\int_{\partial\mathcal U_k}\partial_{y_2}\varphi^*(y)\,\nu_2\qquad\forall\,t >0,
\end{align*}
where $\nu=(\nu_1,\nu_2)$ is the outer unit normal to $\partial\mathcal U_k$.
Observe that, since $\varphi(0)=0$
and $\partial\varphi(0)=K$, it follows that
$\varphi^*\geq 0$ and $\varphi^*|_K=0$. Thus, since $K=[-1,1]\times \{0\}$ and $\varphi^*$ is 1-Lipschitz (see \eqref{eq:1Lip}), 
$$
0 \leq \varphi^*(y_1,y_2)\leq |y_2|\qquad \forall\,y_1 \in [-1,1],
$$
and \cite[Corollary A.23]{figalliBook}
applied to the convex function $\varphi^*(\cdot,y_2)$ yields
$$
|\partial_{y_1}\varphi^*(y_1,y_2)| \leq 2|y_2|\qquad \forall\,y_1 \in [-1/2,1/2].
$$
Thanks to this estimate, since $|y_2|\leq 2^{-k}$ on $\partial\mathcal U_k$ we can bound
$$
\int_{\partial\mathcal U_k}\partial_{y_1}\varphi^*(y)\,\nu_1 \leq 2\cdot 2^{-k}\int_{\partial\mathcal U_k}|\nu_1|=2^{-2k},
$$
thus
$$
2\cdot(2\pi \lambda_1)^{1/2}\int_{\mathcal U_k}|\nabla \varphi^*(y)|^{1/2}\,dy\leq t\,2^{-2k}+\frac{1}{t}\int_{\partial\mathcal U_k}\partial_{y_2}\varphi^*(y)\,\nu_2\qquad\forall\,t >0.
$$
Note that, by the monotonicity of the gradient of convex functions,
$$
\partial_{y_2}\varphi^*(y_1,2^{-k})\geq \partial_{y_2}\varphi^*(y_1,2^{-(k+1)})\qquad \forall\,y_1\quad \Rightarrow\quad \int_{\partial\mathcal U_k}\partial_{y_2}\varphi^*(y)\,\nu_2\geq 0.
$$
Thus, choosing $t:=2^{k}\left(\int_{\partial\mathcal U_k}\partial_{y_2}\varphi^*(y)\nu_2\right)^{1/2}$ we obtain
$$
(2\pi \lambda_1)^{1/2}\int_{\mathcal U_k}|\nabla \varphi^*(y)|^{1/2}\,dy\leq  2^{-k}\left(\int_{\partial\mathcal U_k}\partial_{y_2}\varphi^*(y)\,\nu_2\right)^{1/2},
$$
or equivalently
$$
\frac{\pi}{2} \lambda_1\left(2^{k+1}\int_{\mathcal U_k}|\nabla \varphi^*(y)|^{1/2}\,dy\right)^2\leq \int_{\partial\mathcal U_k}\partial_{y_2}\varphi^*(y)\,\nu_2.
$$
Summing over $k$ the above inequalities and noticing that the boundary integrals appearing in the right hand side form a telescopic series, recalling \eqref{eq:1Lip} we conclude that
\begin{equation}
\label{eq:sum k}
\frac{\pi}{2} \lambda_1\sum_{k\geq 1}
\left(2^{k+1}\int_{\mathcal U_k}|\nabla \varphi^*(y)|^{1/2}\,dy\right)^2 \leq \int_{-1/2}^{1/2}\partial_{y_2}\varphi^*(y_1,1/2)\,dy_1\leq 1.
\end{equation}
We now want to obtain a contradiction by showing that the series in the left hand side diverges.

To this aim notice that, by the convexity of $\varphi^*$ and because $\varphi^*|_K=0$, \eqref{eq:varphi*} implies
$$
|\nabla\varphi^*(y_1,y_2)|\geq \partial_{y_2}\varphi^*(y_1,y_2)\geq \frac{\varphi^*(y_1,y_2)}{y_2}\geq \frac{\hat c_R}{|\log y_2|}\geq \frac{\hat c_R}{k}\qquad \forall\,(y_1,y_2)\in \mathcal U_k.
$$
Recalling \eqref{eq:sum k}, we conclude that
$$
1\geq  \frac{\pi}{2}\lambda_1 \hat c_R\sum_{k\geq 1}\frac{1}{k}=+\infty,
$$
a contradiction.

This proves that $K=\partial\varphi(0)$ must be a point,
and recalling \cite[Lemmata A.21 and A.24]{figalliBook} we obtain that both $\nabla\varphi$ and $\nabla \varphi^*$ are continuous, thus
${\bf Q}_{\tiny{\pm}}:B_1\to \R^2$ is a homeomorphism, as desired.
\end{proof}

\bigskip
{\it Acknowledgments:} The author is extremely grateful to Marc Hallin, both for proposing this problem to him and for several stimulating discussions during the preparation of this manuscript.
The author is thankful to Connor Mooney for useful comments on a preliminary version of the manuscript.
The author is supported by the ERC Grant ``Regularity and Stability in Partial Differential Equations (RSPDE)''

\end{document}